\documentclass[pdflatex,sn-mathphys-num]{sn-jnl}
\usepackage[T1]{fontenc}
\usepackage[utf8]{inputenc}
\usepackage{amsmath,amssymb,amsthm}
\usepackage{enumerate}
\usepackage{color}
\hypersetup{hidelinks,pdftitle={The spherical cap conjecture for immersed disks},
            pdfauthor={Jos\'e M. Espinar}}

\theoremstyle{thmstyleone}
\newtheorem{theorem}{Theorem}[section]
\newtheorem{lemma}[theorem]{Lemma}

\theoremstyle{thmstylethree}
\newtheorem{remark}[theorem]{Remark}

\newcommand{\R}{\mathbb{R}}
\newcommand{\Hq}{\mathbb{H}}
\newcommand{\Sf}{\mathbb{S}}
\newcommand{\Id}{\mathrm{Id}}
\newcommand{\SU}{\mathrm{SU}(2)}
\newcommand{\SO}{\mathrm{SO}(3)}
\newcommand{\tr}{\mathrm{tr}}
\newcommand{\ii}{\mathbf{i}}
\newcommand{\jj}{\mathbf{j}}
\newcommand{\kk}{\mathbf{k}}

\newcommand{\A}{\mathbf{A}}
\begin{document}

\title{The spherical cap conjecture for immersed disks}
\author*{\fnm{Jos\'e M.} \sur{Espinar}}\email{jespinar@ugr.es}
\affil{\orgdiv{Departamento de Geometr\'ia y Topolog\'ia and Instituto de
Matem\'aticas (IMAG)}, \orgname{Universidad de Granada},
\orgaddress{\city{Granada}, \postcode{18071}, \country{Spain}}}


\abstract{We prove that every smooth immersed disk in Euclidean three-space with
nonzero constant mean curvature, regular up to the boundary and mapping its
boundary diffeomorphically onto a circle, is an embedded spherical cap.}


\maketitle

\begingroup\raggedright
\bmhead{Acknowledgements} Jos\'e M. Espinar is partially supported by MINECO/\allowbreak MICINN/\allowbreak FEDER grant no PID2024-\allowbreak160586NB-I00 and by the ``Maria de Maeztu'' Excellence Unit IMAG, reference CEX2020-001105-M, funded by \allowbreak MCINN/AEI/10.13039/501100011033/CEX2020-001105-M.\par
\endgroup

\section{Introduction}

The Spherical Cap Conjecture goes back to a problem of R. Gulliver and
R. Kusner at the 1984 Arcata conference, recorded as Problem 1.4 of the
collection edited by J. Brothers \cite{Brothers1986}, and it was formulated
explicitly by R. S\'a Earp, F. Brito, W. H. Meeks III and H. Rosenberg
\cite{SBMR}. In its immersed form it reads:

\begin{quote}
{\bf The Spherical Cap Conjecture.} \emph{Spherical caps are the only
constant mean curvature disks immersed in $\R^3$ with $H\neq0$ and circular
boundary.}
\end{quote}

We prove it.

\begin{quote}
{\bf Theorem A.} \emph{Let $\Sigma$ be a compact disk and let
$X:\Sigma\to\R^3$ be a smooth immersion with constant mean curvature
$H\neq0$, regular up to the boundary, mapping $\partial\Sigma$
diffeomorphically onto a circle $\Gamma$. Then $X$ is an embedding onto a
closed cap of a sphere of radius $1/|H|$.}
\end{quote}

The disk hypothesis is not a technical one. N. Kapouleas
\cite[Theorem 4.3 and Example 4.4]{Kapouleas1991} constructed, for every
genus $g\geq3$ and every radius $r\in(0,1)$, infinitely many compact
immersed constant mean curvature surfaces of genus $g$ with $H\equiv1$
whose boundary is a round planar circle of radius $r$. According to
\cite[Section 2]{Lopez2025} these examples have self-intersections and lie
in one of the half-spaces determined by the boundary plane. So the immersed
problem without the disk hypothesis is false, and it is false in every
genus $g\geq3$. The conjecture for compact embedded surfaces of arbitrary
genus, also stated in \cite{SBMR}, remains open.

E. Heinz \cite{Heinz1969} proved that necessarily $|H|r\leq1$ (see also
\cite[Corollary 5.1.8]{Lopez2013}), and the extremal case $|H|r=1$ forces
the hemisphere, with no hypothesis on the topology
\cite[Corollary 5.1.11]{Lopez2013}, that case having been solved earlier by
F. Brito and R. S\'a Earp \cite{BritoSaEarp1991}. J. L. Barbosa
\cite{Barbosa1990} proved that such a surface contained in a closed ball of
radius $1/|H|$ is a spherical cap, and obtained the same conclusion in a
later paper for a surface contained in a closed solid right cylinder of
radius $1/|H|$; see \cite[Theorems 5.3.1 and 5.3.5]{Lopez2013}, where
neither statement carries a hypothesis on the topology. R. L\'opez and
S. Montiel \cite{LopezMontiel1995} proved that a disk as above has area at
least that of the small cap, with equality only for the small cap, and,
combining this with the isoperimetric inequality of J. L. Barbosa and
M. do Carmo \cite{BarbosadoCarmo}, that a disk whose area does not exceed
the area of the large cap is a spherical cap; see also
\cite{LopezMontiel1996}. P. A. Hinojosa \cite{Hinojosa2006} sharpened those
area and volume estimates and obtained the spherical cap under an
additional area or volume bound. L. J. Al\'ias, R. L\'opez and B. Palmer
\cite{ALP} proved that a \emph{stable} disk is a planar disk or a spherical
cap. See \cite{Lopez2025} for a recent survey.

Throughout, $\Sigma$ is a compact disk and $X:\Sigma\to\R^3$ is a smooth
immersion with constant mean curvature $H$, regular up to the boundary,
mapping $\partial\Sigma$ diffeomorphically onto a circle $\Gamma$ of radius
$r$ contained in a plane $P$. Let $I$ be the induced metric and let $J$ be
the rotation by $\pi/2$ in $T\Sigma$. We use $dN=-dX\circ S$ and
$H=\frac12\tr S$ throughout, so that a sphere of radius $R$ with its
outward normal has $H=-1/R$; see \eqref{eq:conventions}. Up to a homothety, 
we can assume $H=1$ without loss of generality. 

The proof goes as follows. We identify $\R^3$ with the imaginary quaternions. For constant $H$ the
one-form $\beta:=H\,dX\circ J$ is flat, so on a disk it develops into a map
$q:\Sigma\to\Sf^3$ with $dq=q\beta$, single valued and smooth up to and
including the boundary; for $H\neq0$ the map $q/H$ is the minimal cousin of
$X$ in $\Sf^3(1/|H|)$, and this is the Lawson correspondence
\cite[Section 12]{Lawson1970} in the quaternionic form of
K. Gro\ss e-Brauckmann, R. B. Kusner and J. M. Sullivan
\cite[Theorem 1.1 and Proposition 1.2]{GKS}.  Along $\Gamma$, the 
developing equation becomes an ordinary differential
equation in $\Sf^3$ driven by the contact angle. The change of frame in
Lemma~\ref{lem:boundary} produces an auxiliary curve whose energy is
fixed by the flux at $\pi^2/L$, where $L=2\pi r$. On a disk the developing
map closes, so the auxiliary curve joins antipodal points of the unit
sphere. Its energy is therefore the least possible for those endpoints,
and equality forces constant speed and hence a constant contact angle.
Thus $\Gamma$ is a line of curvature of $\Sigma$, and Nitsche's theorem
gives the conclusion.

The paper is organized as follows. Section \ref{sec:prelim} fixes the
notation and recalls what we use of the Lawson correspondence and of the
flux formula. Section \ref{sec:boundary} contains the boundary lemma, which
is the heart of the paper. Section \ref{sec:cap} passes from the boundary to the 
cap and proves Theorem A.

\section{Preliminaries}\label{sec:prelim}

In this section we fix the notation and recall the two facts we use,
namely the Lawson correspondence in its quaternionic first order form and
the flux formula.

\subsection{Quaternions}
We use the standard identification of $\R^4$ with the quaternions,
\[
 \Hq=\R\oplus\mathrm{Span}\{\ii,\jj,\kk\}=\R\oplus\Im(\Hq),
\]
where ${\bf 1}=(1,0,0,0)$, $\ii=(0,1,0,0)$, $\jj=(0,0,1,0)$ and
$\kk=(0,0,0,1)$, and we identify $\Im(\Hq)$ with $\R^3$ in the usual
manner, that is, $\ii,\jj,\kk$ identify with the canonical basis of $\R^3$
by removing the real part; in particular $\ii\jj=\kk$. For $p,q\in\Hq$ we
have as usual:
\begin{enumerate}[(Q1)]
\item Decomposition into real and imaginary parts:
 $p=\Re(p)+\Im(p)$.
\item Product of quaternions:
 \[
  p\cdot q=\bigl(\Re(p)\Re(q)-\langle\Im(p),\Im(q)\rangle\bigr)
  +\bigl[\Re(p)\Im(q)+\Re(q)\Im(p)+\Im(p)\times\Im(q)\bigr],
 \]
 hence $u\cdot v=-\langle u,v\rangle+u\times v$ and
 $\langle u,v\rangle=-\Re(u\cdot v)$ for all $u,v\in\Im(\Hq)$.
\item Conjugate of a quaternion: $\overline p=\Re(p)-\Im(p)$.
\item Norm: $\|p\|=\sqrt{\overline p\cdot p}$, and hence
 $\|p\cdot q\|=\|p\|\,\|q\|$.
\item $u\cdot v-v\cdot u=2\,u\times v$ for all $u,v\in\Im(\Hq)$.
\item $\langle p\cdot u,q\rangle+\langle p,q\cdot u\rangle=0$ for all
 $p,q\in\Hq$ and all $u\in\Im(\Hq)$.
\end{enumerate}

\subsection[The sphere S3 and SU(2)]{The sphere $\Sf^3$ and $\SU$}
We use the standard model of the three dimensional sphere inside $\Hq$ as
the unit quaternions,
\[
 \Sf^3=\{p\in\Hq:\ \langle p,p\rangle=1\},
\]
endowed with its round metric of constant sectional curvature one, given
by the restriction to $\Sf^3$ of the Euclidean product of
$\R^4\equiv\Hq$. Thus $(\Sf^3,\cdot)$ is a Lie group with identity element
${\bf 1}$, and $T_{{\bf 1}}\Sf^3=\Im(\Hq)$ identifies with its Lie algebra. The
identities
\[
 \langle p\cdot q,p\cdot w\rangle=\|p\|^2\langle q,w\rangle
 =\langle q\cdot p,w\cdot p\rangle,\qquad p,q,w\in\Hq,
\]
tell us that, given $p\in\Sf^3$, the left and right translations
$l_p(q)=p\cdot q$ and $r_p(q)=q\cdot p$ are isometries of
$(\Hq,\langle\cdot,\cdot\rangle)$; in particular the round metric of
$\Sf^3$ is bi-invariant. We identify $\Sf^3$ with $\SU$ in the usual 
way. We write ${\bf 1}$ throughout for the identity quaternion,
which is the identity element of $\Sf^3$ and the unit of $\Hq$. We 
also omit the product sign $\cdot $ when no confusion occurs.

\subsection{Conventions for the immersion}
Let $\Sigma$ be a compact connected oriented surface with boundary and let
$X:\Sigma\to\R^3$ be a smooth immersion, regular up to $\partial\Sigma$,
with unit normal $N$, induced metric $I$ and shape operator $S$. We use
\begin{equation}\label{eq:conventions}
 dN=-dX\circ S,\qquad H=\tfrac12\tr S,\qquad dX(JU)=N\times dX(U),
\end{equation}
where $J$ is the rotation by $\pi/2$ in $T\Sigma$ determined by $I$ and by
the orientation. Thus $\Delta_IX=2HN$, and a sphere of radius $R$ with its
outward normal has $H=-1/R$. Regular up to the boundary means that $dX$
has rank two at every point of $\Sigma$, boundary points included.

Let $t$ be the unit tangent of $\partial\Sigma$ for the orientation
induced by that of $\Sigma$, and set $T=dX(t)$ and $\nu=dX(Jt)$, so that
$\{t,Jt\}$ is a positively oriented orthonormal basis of $T\Sigma$ along
$\partial\Sigma$. Then $\nu$ is the inward conormal and $N=T\times\nu$.

\subsection{The developing map}
Let $H\neq0$ be constant and set
\begin{equation}\label{eq:beta}
 \beta:=H\,dX\circ J,
\end{equation}
a one-form on $\Sigma$ with values in $\Im(\Hq)$. We look for a map
$q:\Sigma\to\Hq$ solving the first order system
\begin{equation}\label{eq:cousin}
 dq=q\cdot\beta ,
\end{equation}
whose integrability condition is $d\beta+\beta\wedge\beta=0$. Without
assuming that $H$ is constant, in positively oriented conformal
coordinates $z=u+iv$ one has $\beta=H(X_v\,du-X_u\,dv)$ and
$X_{uu}+X_{vv}=2HX_u\times X_v$, so that
\[
 d\beta=-\bigl(2H^2X_u\times X_v+H_uX_u+H_vX_v\bigr)du\wedge dv,
 \qquad
 \beta\wedge\beta=2H^2X_u\times X_v\,du\wedge dv
\]
for the convention
$(\beta\wedge\beta)(\partial_u,\partial_v)=\beta_u\beta_v-\beta_v\beta_u$,
where the second identity uses (Q5). Therefore
$d\beta+\beta\wedge\beta=-(H_uX_u+H_vX_v)\,du\wedge dv$, and
\eqref{eq:cousin} is integrable exactly when $H$ is constant.

We now normalize $H$. A homothety of $\R^3$ of ratio $\lambda>0$ replaces
$H$ by $H/\lambda$, and reversing the orientation of $\Sigma$ replaces $N$
by $-N$, $J$ by $-J$ and $H$ by $-H$; in both cases $\beta=H\,dX\circ J$ is
unchanged, and with it the developing map. A reversal
also reverses the induced orientation of $\partial\Sigma$, and a rotation
by $\pi$ about a diameter of $\Gamma$, which is orientation preserving,
restores the parametrization \eqref{eq:circle}. We may and do assume
therefore that
\begin{equation}\label{eq:normalization}
 H=1,\qquad\hbox{so that }\ \beta=dX\circ J ,
\end{equation}
the statements of the introduction, which are for an arbitrary constant
$H\neq0$, following by scaling back. 

By the Frobenius theorem, every choice of an initial condition
$q(p_0)=q_0$ determines a unique solution of \eqref{eq:cousin} on a simply
connected domain, and two solutions differ by a constant left factor. By
(Q6),
\[
 d\langle q,q\rangle=2\langle q\cdot\beta,q\rangle=0
\]
because $\beta$ takes values in $\Im(\Hq)$, so $\|q\|$ is constant on
$\Sigma$; since $q(p_0)\in\Sf^3$, the solution takes values in $\Sf^3$,
and it is unique up to left translations in $\Sf^3$. We call $q$ the
developing map of $X$.

The map $q$ is itself a minimal isometric immersion into the unit sphere
$\Sf^3$, with unit normal $\widetilde N=q\cdot N$. Indeed
$dq(U)=q\cdot dX(JU)$ by \eqref{eq:cousin}, $J$ is an isometry of $I$ and
left translations are isometries of $\Hq$, so $q$ is isometric; and $\langle q\cdot N,q\cdot
dX(JU)\rangle=\langle N,dX(JU)\rangle=0$ together with $\langle q\cdot
N,q\rangle=\langle N,{\bf 1}\rangle=0$ identifies the unit normal. This is the
Lawson correspondence \cite[Section 12]{Lawson1970} in the quaternionic
form of \cite[Theorem 1.1 and Proposition 1.2]{GKS}.

The main geometric objects associated with $X$ and with its cousin are
related as follows, with $\widetilde S$ the shape operator of $q$
determined by $d\widetilde N=-dq\circ\widetilde S$, and with $1\pm\mu$,
$\mu\geq0$, the principal curvatures of $X$:
\begin{eqnarray}
 \mbox{Unit normals:} & \widetilde N=q\cdot N .&\label{eq:table-normal}\\
 \mbox{Shape operators:} & J\circ\widetilde S=S-\Id .&
  \label{eq:table-shape}\\
 \mbox{Principal curvatures:} & \widetilde\kappa_\pm=\pm\mu .&
  \label{eq:table-curvatures}\\
 \mbox{Principal directions:} & \hbox{those of $\widetilde S$ bisect
  those of $S$.}&\label{eq:table-directions}
\end{eqnarray}

The first-order formulation of the Lawson correspondence
\cite[Theorem 1.1]{GKS} gives the following boundary extension by smooth
dependence on the initial data and parameters.

\begin{lemma}\label{lem:extension}
If $X$ is smooth up to a boundary arc, then $q$ extends smoothly to that
arc. In particular, on a compact disk $q$ is single valued and smooth on
the whole closed disk, so its restriction to the boundary is a closed
curve in $\SU$, and not only in $\SO$.
\end{lemma}

\subsection{The boundary data}
Whenever the boundary is circular we normalize as follows. After an
orientation preserving rigid motion of $\R^3$ we may and do assume that
$X(\partial\Sigma)$ is the circle $\Gamma$ of radius $r$ in the plane
$\{\langle\cdot,\kk\rangle=0\}$, parametrized by its arclength
$s\in[0,L]$, $L=2\pi r$, as
\begin{equation}\label{eq:circle}
 \gamma(s)=r\Bigl(\cos\frac sr,\sin\frac sr,0\Bigr),
 \qquad e_r=\cos\frac sr\,\ii+\sin\frac sr\,\jj,
 \qquad T=-\sin\frac sr\,\ii+\cos\frac sr\,\jj,
\end{equation}
with $T$ the image under $dX$ of the positively oriented unit tangent of
$\partial\Sigma$. Since
$X|_{\partial\Sigma}$ is a diffeomorphism onto $\Gamma$ and $dX$ is an
isometry, $s$ is at the same time the $I$-arclength of $\partial\Sigma$.
Let us write
\begin{equation}\label{eq:ab}
 \nu=a\,e_r+b\,\kk,\qquad N=b\,e_r-a\,\kk,\qquad a^2+b^2=1,
\end{equation}
and locally $a=\cos\alpha$, $b=\sin\alpha$. The Darboux equations
$T'=k_g\nu+k_nN$ and $\nu'=-k_gT+\tau_gN$ then give
\begin{equation}\label{eq:darboux}
 k_g=-\frac ar,\qquad k_n=-\frac br,\qquad \tau_g=-\alpha' .
\end{equation}
Note that $a=\langle\nu,e_r\rangle$ measures the angle at which $\Sigma$
meets the plane of $\Gamma$, and $\langle N,\kk\rangle=-a$.

\subsection{The flux}
The first identity below is the balancing formula of N. Korevaar,
R. Kusner and B. Solomon \cite{KKS} in the form used for immersions by
R. L\'opez and S. Montiel \cite{LopezMontiel1996}, and for a circular
boundary it is in \cite{BritoSaEarp1991}; see also
\cite[Section 3.1]{BranderLopez} and \cite[Corollary 5.1.8]{Lopez2013}. 
With our choice of inward conormal, the identities recalled in these
references take the following form.

\begin{lemma}\label{lem:flux}
Let $\Sigma$ be compact connected oriented and let $X$ be as in
\S\ref{sec:prelim}, normalized as in \eqref{eq:normalization}. Then
\begin{equation}\label{eq:flux}
 \int_{\partial\Sigma}\nu\,ds=-2\,\A,
 \qquad \A:=\frac12\int_{\partial\Sigma}X\times dX=\int_\Sigma N\,dA .
\end{equation}
If moreover $\partial\Sigma$ is connected, $X(\partial\Sigma)$ is the
circle \eqref{eq:circle} and $a,b$ are as in \eqref{eq:ab}, then
\begin{equation}\label{eq:scalarflux}
 \int_0^Lb\,ds=-rL,\qquad\hbox{and consequently}\qquad r\leq1 .
\end{equation}
\end{lemma}

Together with \eqref{eq:darboux} they give the equivalences we shall use
without further comment: for a circular boundary,
\begin{equation}\label{eq:equivalences}
 \hbox{$\Gamma$ is a line of curvature}
 \iff \hbox{$a$ is constant}
 \iff \hbox{$b$ is constant}
 \iff k_n\equiv1 .
\end{equation}
Indeed $\tau_g\equiv0$ means $\alpha'\equiv0$; $a^2+b^2=1$ and continuity
make constancy of $a$ and of $b$ equivalent; and if $b$ is constant then
\eqref{eq:scalarflux} forces $b=-r$, that is, $k_n\equiv1$. The first
equivalence is Joachimsthal's theorem.

\section{The boundary lemma}\label{sec:boundary}

Along the boundary the developing equation \eqref{eq:cousin} becomes an
ordinary differential equation in $\Sf^3$, $q'=q\nu$ where $\nu=a e_r+b\kk$ 
and $|\nu|=1$, driven by the pair $(a,b)$ of
\eqref{eq:ab} alone. 

Two observations are in order. First, the solutions of
$q'=q\nu$ have unit speed, $|q'|=|q|\,|\nu|=1$ by (Q4) and
\eqref{eq:ab}, whatever the contact angle, because left translation is an
isometry of $\Hq$; so the pair $(a,b)$ does not enter the speed of $q$.
Second, the data are written in the rotating frame $\{e_r,T,\kk\}$,
which makes one full turn about $\kk$ as $s$ runs over $[0,L]$. 
Conjugation by $E(s)=\exp\bigl(s\kk/(2r)\bigr)$ is the rotation
carrying $\ii$ to $e_r(s)$ and fixing $\kk$; the half angle comes
from the double cover $\Sf^3\to\SO$.
Since $\kk$ commutes with $E$, setting $p=qE$ replaces $e_r$
by $\ii$ in the equation and adds the constant term $\kk/(2r)$.
This term is the generator $E^{-1}E'$, whereas the corresponding
frame rotates with angular velocity $\kk/r$. The speed of $p$ 
need not be constant: it is given by \eqref{eq:speed} below, 
and the term in $b$ is the one the flux controls.

This frame returns to itself after one turn, but its lift does not:
$E(L)=\exp(\pi\kk)=-{\bf 1}$. If the developing map closes up,
$q(L)=q(0)$, then $p$ joins a point of $\Sf^3$ to its antipode, so its
length is at least $\pi$, while the flux fixes its energy at $\pi^2/L$. By
Cauchy--Schwarz the length is also at most $\pi$, and the equality case
gives the conclusion. 

\begin{lemma}\label{lem:boundary}
Let $r>0$ and $L=2\pi r$. Let $a,b:[0,L]\to\R$ be continuous with
\[
 a^2+b^2=1,\qquad\int_0^Lb\,ds=-rL,
\]
and let $q:[0,L]\to\Sf^3$ solve $q'=q(a\,e_r+b\,\kk)$. Then $q(L)=q(0)$ if
and only if $b\equiv-r$ and $a$ is constant.
\end{lemma}

\begin{proof}
Set
\begin{equation}\label{eq:gauge}
 E(s):=\exp\Bigl(\frac{s\kk}{2r}\Bigr),\qquad p:=qE .
\end{equation}

Then
\[
 E\ii E^{-1}=e_r,\qquad E\kk=\kk E,\qquad
 E'=E\frac{\kk}{2r},\qquad
 E(0)={\bf 1},\quad E(L)=-{\bf 1}.
\]
In particular, $E$ is defined on the interval $[0,L]$, not on the
boundary circle. Since $(a e_r+b\kk)E=E(a\ii+b\kk)$, the product rule
yields
\[
 p'=q'E+qE'=qE(a\ii+b\kk)+qE\frac{\kk}{2r}.
\]
Hence
\begin{equation}\label{eq:pprime}
 p'=p\Omega,\qquad
 \Omega(s)=a\ii+\Bigl(b+\frac1{2r}\Bigr)\kk.
\end{equation}
Using $|p|=1$ and $a^2+b^2=1$, we obtain
\begin{equation}\label{eq:speed}
 |p'|^2=|\Omega|^2
 =a^2+\Bigl(b+\frac1{2r}\Bigr)^2
 =1+\frac br+\frac1{4r^2}.
\end{equation}
The prescribed integral of $b$ \eqref{eq:scalarflux} now gives
\begin{equation}\label{eq:energy}
 \int_0^L|p'|^2\,ds=\frac{\pi^2}{L}.
\end{equation}
This identity does not yet use the closure of $q$. Assume now that $q(L)=q(0)$. 
Since $E(0)={\bf 1}$ and $E(L)=-{\bf 1}$,
we have $p(L)=-p(0)$. Antipodal points of the unit sphere have intrinsic
distance $\pi$. Therefore \eqref{eq:energy} and Cauchy--Schwarz give
\begin{equation}\label{eq:saturation}
 \pi\leq\int_0^L|p'|\,ds
 \leq\Bigl(L\int_0^L|p'|^2\,ds\Bigr)^{1/2}=\pi.
\end{equation}
Equality in Cauchy--Schwarz forces $|p'|\equiv\pi/L=1/(2r)$, and
\eqref{eq:speed} then gives $b\equiv-r$. Thus $a^2=1-r^2$, and
continuity on the connected interval $[0,L]$ implies that $a$ is
constant.

Conversely, if $b\equiv-r$ and $a$ is constant, then $\Omega$ is a
constant imaginary quaternion with $ |\Omega|^2=\frac1{4r^2}$.
Hence $p(s)=p(0)\exp(s\Omega)$ and $L|\Omega|=\pi$, so
$\exp(L\Omega)=-{\bf 1}$. Consequently $p(L)=-p(0)$ 
and $q(L)=p(L)E(L)^{-1}=(-p(0))(-{\bf 1})=q(0)$.
\end{proof}

\begin{remark}\label{rem:fourier}
There is a second proof of the equality step. Assume again that $q(L)=q(0)$, 
so that  $p(L)=-p(0)$. Extend $p$ to $[0,2L]$ by $p(s+L)=-p(s)$; the 
extension  is an $H^1$ map of the circle of length $2L$ into $\R^4$ with 
only odd  frequencies, so the antiperiodic Wirtinger inequality gives 
$$\int_0^L|p'|^2\,ds\geq(\pi/L)^2\int_0^L|p|^2\,ds=\pi^2/L .$$

By \eqref{eq:energy} this is an equality, so $p(s)=A\cos(\pi s/L)+B\sin(\pi s/L)$ 
with $|A|=|B|=1$ and $\langle A,B\rangle=0$, and the speed is constant.
\end{remark}

\section{Proof of Theorem A}\label{sec:cap}

Let $\Sigma$ and $X$ be as in Theorem A, normalized as in
\eqref{eq:normalization}. Let $s$ be the arclength of $\Gamma$, which by \S\ref{sec:prelim} is also
the $I$-arclength of $\partial\Sigma$. Along the boundary $t$ is the unit
tangent, so
\[
 \beta(t)=dX(Jt)=\nu=a\,e_r+b\,\kk,
\]
and \eqref{eq:cousin} reads $q'=q(a\,e_r+b\,\kk)$, which is the equation of
Lemma \ref{lem:boundary}. Since $\Sigma$ is a disk, $q$ is single valued and
smooth on the whole closed disk by Lemma \ref{lem:extension}, so
$q(L)=q(0)$. The integral hypothesis is \eqref{eq:scalarflux}. Lemma
\ref{lem:boundary} therefore gives $b\equiv-r$ and $a$ constant, and
\eqref{eq:equivalences} implies that $\Sigma$ meets the plane $P$ along $\Gamma$
at a constant angle. Equivalently, $\Gamma$ is a line of curvature of
$\Sigma$; equivalently, the normal curvature of $\Gamma$ in $\Sigma$ is
identically one. By Nitsche's classical line-of-curvature argument
\cite[pp.~11--12]{Nitsche1985} (see also
\cite[Theorem~5.1.10]{Lopez2013}), regularity up to the boundary and the 
single covering of $\Gamma$, $X$ is an embedding onto a closed spherical cap.

\section*{Use of AI tools}
During the preparation of this work, the author used Claude/Code Fable 5.1 
(Anthropic) and ChatGPT/Codex Astra 6 (OpenAI) to assist with mathematical 
development, with the checking of proof arguments, with language editing, 
and with the preparation of \LaTeX{} code. 

{\bf How this paper came about, and the use of AI:} The problem reached 
the author through H. Rosenberg, who described long ago a candidate 
counterexample to the Spherical Cap Conjecture, and the author set
out to construct it.  LLMs, asked to examine that construction, produced instead 
a condition which would forbid such an example, formulated in terms of the 
monodromy of the developing map. The author recognized the monodromy as 
the holonomy of the Lawson cousin, and it is in that form that it became the 
boundary lemma. 

The analytic boundary lemma (Lemma~\ref{lem:boundary}) has also been
formalized in Lean~4 with Mathlib. The source supplied as LEAN Resource proves both
implications under the stated analytic hypotheses. This formalization
does not include the geometric reduction or the final surface classification.

The author is responsible for the mathematical content and for the final text.

% Add Statements and Declarations / Competing interests after the author supplies the factual declaration.
% See ../DECLARACIONES-PENDIENTES.md.

\end{document}